\documentclass[leqno,sumlimits,intlimits,namelimits,draft=false]{amsart}

\usepackage{amssymb,latexsym}
\usepackage[numeric,initials,bibtex-style,nobysame]{amsrefs}
\usepackage{amsthm}

\usepackage{a4wide}

\usepackage{eucal}
\usepackage{mathrsfs}

\newcommand{\dd}{\mathrm{d}}

\newcommand{\D}{\ensuremath{\mathcal{D}}}

\newcommand{\loc}{\ensuremath{\text{loc}}}

\newcommand{\mb}[1]{\ensuremath{\mathbb{#1}}}

\newcommand{\R}{\mb{R}}

\newcommand{\sgn}{\mathop{\mathrm{sgn}}}

\renewcommand{\d}{\ensuremath{\partial}}
\newcommand{\diff}[1]{\frac{d}{d#1}}

\newfont{\bl}{msbm10 scaled \magstep2}

\newtheorem{theorem}{Theorem}[section]
\newtheorem{lemma}[theorem]{Lemma}
\newtheorem{proposition}[theorem]{Proposition}
\newtheorem{definition}[theorem]{Definition}

\theoremstyle{definition}
\newtheorem{remark}[theorem]{Remark}
\newtheorem{example}[theorem]{Example}

\newcommand{\beq}{\begin{equation}}
\newcommand{\eeq}{\end{equation}}

\newcommand{\col}{\colon}

\newcommand{\FT}[1]{\widehat{#1}}

\newcommand{\notmid}{\mid\kern-0.5em\not\kern0.5em}

\newcommand{\norm}[2]{{\| #1 \|}_{#2}}

\newcommand{\Norm}[1]{\norm{#1}{}}

\newcommand{\al}{\alpha}

\newcommand{\de}{\delta}

\newcommand{\vphi}{\varphi}

\newcommand{\la}{\lambda}

\newcommand{\supp}{\mathop{\mathrm{supp}}}

\usepackage{graphicx}

\begin{document}

\pagestyle{plain}

\title{Discontinuous traveling waves as weak solutions to the Fornberg-Whitham equation}

\author{G\"unther H\"ormann}

\address{Fakult\"at f\"ur Mathematik\\
Universit\"at Wien, Austria}

\email{guenther.hoermann@univie.ac.at}




\date{\today}

\begin{abstract} We analyze the weak solution concept for the Fornberg-Whitham equation in case of traveling waves with a piecewise smooth profile function. The existence of discontinuous weak traveling wave solutions is shown by means of analysis of a corresponding planar dynamical system and appropriate patching of disconnected orbits.
\end{abstract}

\maketitle


\section{Basic concepts}

\subsection{Introduction}

The Fornberg-Whitham equation has been introduced as one of the simplest shallow water wave models which are still capable of incorporating wave breaking (cf.\ \cites{Seliger68,Whitham1974,FB78,NaumShish94, ConstantinEscher1998, Haziot2017,GH2018}). The wave height is described by a function of space and time $u \col \R \times [0,\infty[ \to \R$, $(x,t) \mapsto u(x,t)$, we will occasionally write $u(t)$ to denote the function $x \mapsto u(x,t)$. Suppose that an initial wave profile $u_0$ is given as a real-valued function on $\R$. The Cauchy problem for the Fornberg-Whitham equation is
\begin{align}
   \label{FWEqu}  u_t +  u u_x + K \ast u_x  &= 0,\\
     \label{IC} u(x,0) &= u_0(x),
\end{align}
where the convolution is in the $x$-variable only and  
\beq\label{kernel}
  K(x) = \frac{1}{2} e^{-|x|},  
\eeq
which satisfies $(1 - \d_x^2) K = \delta$.

We note that formally applying $1 - \d_x^2$ to \eqref{FWEqu} produces a third order partial differential equation 
$$
     u_t - u_{txx}  - 3 u_x u_{xx} - u u_{xxx} +  u u_x + u_x  = 0,
$$
but we will stay with the above non-local equation which correponds to the original model and is also more suitable for the weak solution concept.

\begin{remark} Note that we follow here in \eqref{FWEqu} the sign convention for the convolution term as used in  \cite[Equation (4)]{FB78} (or also in \cite[Section 13.14]{Whitham1974}), but used a rescaling of the solution by $3/2$ to get rid of an additional constant factor in the nonlinear term. 
\end{remark}

Well-posedness results on short time intervals for (\ref{FWEqu}-\ref{IC}) with spatial regularity according to Sobolev or Besov scales have been obtained in \cites{Holmes16,HolTho17}. For example, in terms of Sobolev spaces these read as follows:
If $s > 3/2$ and $u_0 \in H^s(\R)$, then there exists $T_0 > 0$ such that (\ref{FWEqu}-\ref{IC}) possesses a unique solution $u \in C([0,T_0], H^s(\R)) \cap C^1([0,T_0],H^{s-1}(\R))$; moreover, the map $u_0 \mapsto u$ is  continuous $H^s(\R) \to C([0,T_0], H^s(\R))$ and $\sup_{t \in [0,T_0]} \norm{u(t)}{H^s(\R)} < \infty$.

\subsection{Weak solution concept}

Equation \eqref{FWEqu} can formally be rewritten in the form 
\beq\label{FWEqu2}
  \d_t u +  \d_x \big(\frac{u^2}{2}\big) + K' \ast u  = 0, 
\eeq
which suggests to define weak solutions in the context of locally bounded measurable functions in the following way. 
\begin{definition}
A function $u \in L^\infty_\loc(\R \times [0,\infty[)$ is called a \emph{weak solution} of the Cauchy problem (\ref{FWEqu}-\ref{IC}) with initial value $u_0 \in L^\infty_\loc(\R)$, if 
\begin{multline}\label{weaksol}
  \int_0^\infty \int_{-\infty}^\infty \Big( 
    - u(x,t) \d_t \phi(x,t) - \frac{u^2(x,t)}{2} \d_x \phi(x,t) + 
    \big(K' \ast u(.,t)\big)(x) \phi(x,t) \Big) \dd x \dd t \\ = 
  \int_{-\infty}^\infty u_0(x) \phi(x,0) \dd x 
\end{multline}
holds for every test function $\phi \in \D(\R^2)$.
\end{definition}

\begin{remark} In the current paper we will not discuss uniqueness or well-posedness of general weak solutions, which might also require to introduce the concept of an \emph{entropy solution} $u \in L^\infty_\loc(\R \times [0,\infty[)$ with initial value $u_0 \in L^\infty_\loc(\R)$ in the sense that
\begin{multline*}
  0 \leq \int_0^\infty \int_{-\infty}^\infty \Big( 
    \left| u(x,t) - \la \right|  \d_t \phi(x,t) + \sgn \big(u(x,t) - \la\big) \frac{u^2(x,t) - \la^2}{2} 
    \d_x \phi(x,t)\big)(x) \Big) \dd x \dd t \\
     -  \int_0^\infty \int_{-\infty}^\infty \Big( \sgn \big(u(x,t) - \la \big) 
     \big(K' \ast (u(.,t) - \la)\big)(x)
  \Big) \phi(x,t) \dd x \dd t 
  +   \int_{-\infty}^\infty \left| u_0(x) - \la \right| \phi(x,0) \dd x 
\end{multline*}
holds for every nonnegative test function $\phi \in \D(\R^2)$ and every $\la \in \R$.
We note that this entropy condition implies \eqref{weaksol}, since  for any given $\phi$ we may choose $\la = -r$ and $\la = r$, where $r > 0$ is sufficiently large such that $|u| < r$ holds on $\supp(\phi)$. Thus, every entropy solution is a weak solution of the Cauchy problem (\ref{FWEqu}-\ref{IC}).
\end{remark}  
  
\begin{example}[Peakon as weak solution]

We consider the well-known peakon-type traveling wave \cite{FB78,ZT2010,CLH2012} for the Fornberg-Whitham equation, namely
\beq\label{peakon}
   p(x,t) = \frac{4}{3} \exp({- \frac{1}{2} |x - \frac{4}{3} t|}) = U(x - \frac{4}{3} t),
\eeq
where $U(y) := \frac{4}{3} \exp({- \frac{1}{2} |y|})$ is the profile function. It can be described as the solitary wave of greatest height (see \cite[Section 13.14]{Whitham1974}, \cite[Section 6]{FB78} or the more detailed discussions of traveling solitary-wave solutions of the governing equations for two-dimensional water waves propagating in irrotational flow over a flat bed in \cite{CE2007} and \cite{AFT82}).
We observe that  $$
  p(.,0) = U \in H^s(\R) \quad\Longleftrightarrow\quad s < 3/2, 
$$  
which is easily seen from the fact that $\FT{U}(\xi)$ is proportional to $1/(1 + \xi^2)$ (the precise constants depending on the convention of the Fourier transform). Therefore, $p$ has less spatial regularity than required for the strong solution concept and for the well-posedness result mentioned in the introductory subsection. 

It is easy to see that from the peakon given in \eqref{peakon} we obtain a weak solution with initial value $u_0 = U$: We may calculate directly (e.g., as in \cite[Appendix]{Holmes16})  $(1 - \d_x^2) p= 3 p/4 + 4 \de(x - 4t/3) / 3$ and $(1 - \d_x^2) (p^2) = 32 \de(x - 4t/3) / 9$, which implies that $(1 - \d_x^2)(\d_t p + \d_x (p^2/2)) = \d_t ((1 - \d_x^2) p) + \d_x (1 - \d_x^2) (p^2)/2 = - \d_x p$ holds in the sense of distributions\footnote{But note that we insisted on $p^2$ being carried out as pointwise product of functions prior to differentiation.} on $\R^2$ and therefore, upon applying $(1-\d_x^2)^{-1}$ in the form of spatial convolution with $K$, that \eqref{FWEqu2} holds for $p$ on all of $\R^2$; now putting $u (x,t) := p(x,t) H(t)$ (where $H$ denotes the Heaviside function) and observing that $\d_t \big( p(x,t) H(t)\big) =  p(x,0) \de(t) +  \d_t p(x,t) H(t)$ (by checking the action on a test function) we arrive at $\d_t u + \d_x(u^2/2) + K' \ast u = p(.,0) \de(t)  + (\d_t p + \d_x(p^2/2) + K' \ast p) H(t) = p(.,0) \otimes \de + 0 = u_0 \otimes \de $, which means exactly \eqref{weaksol} when applied to a test function (and noting that $t \geq 0$ in $\supp(u)$ by construction).
\end{example}

The following section is devoted to the construction of traveling wave solutions which are bounded and discontinuous.


\section{Bounded traveling waves with discontinuity}

The possible continuous traveling waves for the Fornberg-Whitham equation have been obtained and classified successfully by means of studying the properties of corresponding ordinary differential equations for the profile function, e.g., in \cite{ZT2008,ZT2009,ZT2010,CLH2012}). In our current attempt to construct a discontinuous bounded traveling wave, we will make use of a similar basic strategy and draw on many ideas from these references. In particular, we have to make a somewhat refined analysis of several steps along the way to a correponding first-order system of ordinary differential equations for  the profile function and its derivative. Finally, we will have to find a correct way for patching up a profile function from two disconnected orbits in the topological dynamics. The inspiration for the whole construction stems from a discussion of traveling waves with shocks for a model of radiating gas, the so-called Rosenau model, given in \cite{KN99}.

The typical Ansatz for a traveling wave solution is $u(x,t) = W(x-ct)$ with a profile function $W \col \R \to \R$ and $c \in \R$. We suppose that $W$ is \emph{piecewise $C^2$} in the following sense
\beq\label{pwsmooth}
   \text{$W$ is a $C^2$ function off $0$ and $W, W', W''$ possess one-sided limits at $0$}
\eeq
and, in addition, we require that there exist $A, B \in \R$ such that 
\beq\label{BC}
       \lim_{\xi \to -\infty} W(\xi) = A \quad\text{and}\quad 
       \lim_{\xi \to +\infty} W(\xi) = B.
\eeq
In particular, $W$ belongs to $L^\infty(\R)$ and $u(x,0) = W(x)$.

\subsection{Traveling waves as weak solutions}

A traveling wave $u$ with piecewise smooth profile function $W$ is a weak solution, if and only if for every test function $\phi$ on $\R^2$ we have (upon a change of variables $\xi = x -c t$ in the integrals on the left-hand side of \eqref{weaksol} and with explicit convolution integral)
\begin{multline*}
  \int_0^\infty \int_{-\infty}^\infty \Big( 
    - W(\xi) \d_2 \phi(\xi + ct,t) - \frac{W^2(\xi)}{2} \d_1 \phi(\xi + c t,t) \Big) \dd \xi \dd t\\
     + \int_0^\infty \int_{-\infty}^\infty \int_{-\infty}^\infty
    K'(z) W(\xi - z) \dd z \,\phi(\xi + ct,t)  \dd \xi \dd t =
  \int_{-\infty}^\infty W(x) \phi(x,0) \dd x .
\end{multline*}
In the first integral we make use of the relation $\d_2 \phi(\xi + ct,t) = \diff{t}\big( \phi(\xi + ct, t)\big) - c \d_1 \phi(\xi + ct, t)$, split the $\xi$-integral into two parts according to $\xi < 0$ and $\xi > 0$, and apply integration by parts. Thus, we obtain 
\begin{multline}\label{weaktravel}
  \int_0^\infty \phi(ct,t) \Big( \frac{W^2(0+) - W^2(0-)}{2} 
      + c \big(W(0-) - W(0+)\big) \Big)\dd t\\ 
      + \int_0^\infty \int_{-\infty}^0  \phi(\xi + ct,t) \Big(
      W(\xi) W'(\xi) - c W'(\xi) + (K' \ast W)(\xi)
      \Big) \dd \xi \dd t \\
      + \int_0^\infty \int_{0}^\infty  \phi(\xi + ct,t) \Big(
      W(\xi) W'(\xi) - c W'(\xi) + (K' \ast W)(\xi)
      \Big) \dd \xi \dd t
      = 0.
\end{multline}
Observe that due to the properties of $W$, the $\xi$-integrals could be re-combined into one integration over $\R$, but this could cause a misunderstanding about the exact meaning of the differential equation we want to extract from the above condition. First we note the intermediate result.
\begin{proposition} A traveling wave $u$ with piecewise smooth profile function $W$ (in the sense of \eqref{pwsmooth}) is a weak solution, if and only if  \eqref{weaktravel} holds for every $\phi \in \D(\R^2)$.
\end{proposition}

We note that the linear span of test functions of the form $\phi(x,t) = \vphi_1(x - ct) \vphi_2(t)$ with $\vphi_1, \vphi_2 \in \D(\R)$ is a dense subspace of $\D(\R^2)$, hence \eqref{weaktravel} may be reduced to 
\begin{multline*}
  \vphi_1(0) \int_0^\infty \vphi_2(t) \Big( \frac{W^2(0+) - W^2(0-)}{2} 
      + c \big(W(0-) - W(0+)\big) \Big)\dd t\\ 
      + \int_0^\infty \vphi_2(t) \dd t \left( \; \int_{-\infty}^0  \vphi_1(\xi) \Big(
      W(\xi) W'(\xi) - c W'(\xi) + (K' \ast W)(\xi)
      \Big) \dd \xi \dd t \right.\\
      + \left. \int_{0}^\infty  \vphi_1(\xi) \Big(
      W(\xi) W'(\xi) - c W'(\xi) + (K' \ast W)(\xi)
      \Big) \dd \xi \right)   = 0.
\end{multline*}
Choosing $0 \leq \vphi_1 \leq 1$ with support arbitrarily close to $0$ and $\vphi_1(0) = 1$ while letting $\vphi_2$ vary in $\D(\R)$ we deduce $(W^2(0+) - W^2(0-))/2 = c (W(0+) - W(0-))$, which yields the \emph{Rankine-Hugoniot condition}
\beq\label{RH}
   W(0+) + W(0-) = 2 c,
\eeq
if $W(0-) \neq W(0+)$. Having observed this we may now choose $\vphi_2$ such that $\int_0^\infty \vphi_2(t) \dd t = 1$ and $\vphi_1$ with support in $\xi < 0$ or in $\xi > 0$, but otherwise arbitrary, and deduce  
\beq\label{WODE1}
  (W(\xi) - c) W'(\xi) + (K' \ast W)(\xi) = 0 \qquad \forall \xi \neq 0.
\eeq
On the other hand, we see that \eqref{RH} and \eqref{WODE1} together imply \eqref{weaktravel}, which along with the above proposition proves the following statement.

\begin{theorem}\label{weaktravelthm} A traveling wave $u$ with piecewise smooth, but discontinuous, profile function $W$  is a weak solution of the Cauchy problem (\ref{FWEqu}-\ref{IC}) with initial value $u_0 = W$, if and only if $W$ satisfies the Rankine-Hugoniot condition \eqref{RH} and the integro-differential equation \eqref{WODE1}.
\end{theorem}

\begin{remark}
Note that constant functions $u$ are obviously strong solutions to \eqref{FWEqu2}, whereas a  piecewise constant, discontinuous, profile function $W$ cannot produce a weak traveling wave solution $u$: If $W(\xi) = A H(-\xi) + B H(\xi)$  with $A \neq B$, then $W' = (B - A) \de$ and this leads to a contradiction in \eqref{WODE1} due to the convolution term producing $(B - A) K(\xi)$ in this case.
\end{remark}

Now suppose that we have a discontinuous traveling wave solution according to the theorem. We may take the limits $\xi \to 0-$ and $\xi \to 0+$ in Equation \eqref{WODE1}, take the difference of the equations thus obtained, and note that $K' \ast W \in L^1 \ast L^\infty$ is (uniformly) continuous on $\R$ (\cite[14.10.6(ii)]{Dieudonne:V2E}) to deduce
$$
  (W(0+) - c) W'(0+) = (W(0-) - c) W'(0-).
$$
The Rankine-Hugoniot condition \eqref{RH} means $W(0+)  - c =  c - W(0-)$, which by discontinuity of $W$ requires $W(0-) \neq c$ and $W(0+) \neq c$, so that we obtain a  relation for the one-sided derivatives
\beq\label{dercond}
   W'(0+) + W'(0-) = 0
\eeq
as a further necessary condition.

\subsection{Analysis of the integro-differential equation for the traveling wave profile}

We observe that Equation \eqref{WODE1} can be written as
$$
   \Big(\frac{(W - c)^2}{2}\Big)'(\xi) + (K \ast W)'(\xi) = 0 \qquad \forall \xi \neq 0
$$
and we will argue that it may be understood as an equation of distributions gobally on $\R$, if $W$ is supposed to satisfy \eqref{RH}. 

For a piecewise continuous function $f$ on $\R$ denote by $[f]$ the measurable function with $[f](\xi) := f(\xi)$ for every $\xi \neq 0$ and $[f](0) := 0$ and recall that for a piecewise $C^1$ function $g$ on $\R$ we have for its distributional derivative $g' = [g'] + (g(0+) - g(0-)) \cdot \de$ (where $[g']$ uses the value of the pointwise classical derivative $g'(\xi)$ for $\xi \neq 0$). Employing this notation, we obtain thanks to \eqref{RH}
\begin{multline} \label{disder1}
   \Big(\frac{(W - c)^2}{2}\Big)' = \Big[\Big(\frac{(W - c)^2}{2}\Big)'\,\Big] + 
     \frac{(W(0+) - c)^2 - (W(0-) - c)^2}{2} \cdot \de\\ 
     = \Big[\Big(\frac{(W - c)^2}{2}\Big)'\,\Big]  + 
     \frac{\big(W(0+) - W(0-)\big)\big(W(0+) + W(0-) - 2 c\big)}{2} \cdot \de\\ =
     \Big[\Big(\frac{(W - c)^2}{2}\Big)'\,\Big] = \Big[ (W - c) W'\Big].
\end{multline}
Therefore, we have in the sense of distributions on $\R$
\beq\label{firstder}
   0 = \Big(\frac{(W - c)^2}{2}\Big)' + (K \ast W)' = 
      \Big(\frac{(W - c)^2}{2} + K \ast W \Big)',
\eeq
which implies that there is a constant $\al \in \R$ such that
\beq\label{const}
   \frac{(W - c)^2}{2} + K \ast W = \al.
\eeq

Similarly to the reasoning above, but now in addition employing Equation \eqref{dercond} (which was a consequence of \eqref{RH} and \eqref{WODE1}), we obtain 
\begin{multline} \label{disder2}
   \Big(\frac{(W - c)^2}{2}\Big)'' = \Big[ (W - c) W'\Big]'\\ = 
     \Big[ \big((W - c) W'\big)'\Big] + 
     \frac{\big(W(0+) - c\big) W'(0+) - \big(W(0-) - c\big) W'(0-)}{2} \cdot \de\\
     = \Big[ ((W - c) W')'\Big] + \frac{ W'(0+) + W'(0-)}{2} \cdot \de =
     \Big[ ((W - c) W')'\Big] = \Big[ (W')^2 + (W-c) W''\Big],
\end{multline}
which allows us to conclude upon differentiating in Equation \eqref{firstder} that
\beq\label{secondder}
   \Big[ (W')^2 + (W-c) W''\Big] + K'' \ast W = 0.
\eeq
Taking now the difference of the Equations \eqref{const} and \eqref{secondder} and recalling that $K - K'' = \de$ we have the following equation 
\beq\label{secondorderpreliminary}
   \frac{(W - c)^2}{2} -  \Big[ (W')^2 + (W-c) W'' \Big]  + W = \al,
\eeq
which gives a classical second-order differential equation on $\R \setminus \{0\}$.
So far, we have shown the first part of the following

\begin{proposition}\label{propweaksol} For every \emph{discontinuous} profile function $W$ the Rankine-Hugoniot condition \eqref{RH} and the integro-differential equation \eqref{WODE1} imply  Equation \eqref{secondorderpreliminary}. On the other hand,   \eqref{secondorderpreliminary} in combination with conditions \eqref{RH} and \eqref{dercond} implies \eqref{WODE1} and thus, according to Theorem \ref{weaktravelthm}, defines a weak traveling wave solution to the Cauchy problem (\ref{FWEqu}-\ref{IC}).
\end{proposition}

\begin{proof} It remains to prove the second part of the statement. We recall from the details of the above reasoning that conditions \eqref{RH} and \eqref{dercond} imply the equality \eqref{disder2} (while \eqref{RH} implies \eqref{disder1}). Applying this to \eqref{secondorderpreliminary} and using again the fact that $(1 - \d_x^2) K = K - K'' = \de$ we get
\begin{multline*}
     \al = \frac{(W - c)^2}{2} - \left( \frac{(W - c)^2}{2} \right)''   + W =
     (1 - \d_x^2) \frac{(W - c)^2}{2} + (K - K'') \ast W\\
     =  (1 - \d_x^2) \left( \frac{(W - c)^2}{2} + K \ast W \right).
\end{multline*}
Noting that $K \ast \al = \al$ we deduce
$\al = \frac{(W - c)^2}{2} + K \ast W$
and differentiate once to obtain \eqref{WODE1}.
\end{proof}

We now determine the constant $\al$ that appeared for the first time in \eqref{const} with the help of the boundary conditions \eqref{BC}: All the distributional equations above have a pointwise classical meaning in $\R \setminus \{ 0 \}$, hence we may evaluate \eqref{const} at any $\xi < 0$ or $\xi > 0$. Moreover, supposing \eqref{BC} we know that the term $(W(\xi) - c)^2/2$ possesses a limit when $\xi \to -\infty$ or $\xi \to \infty$. A brief inspection of the integral defining the convolution $K \ast W$ and appealing to the theorem of dominated convergence shows that this term also has a limit as $\xi \to -\infty$ or $\xi \to \infty$, namely $A$ or $B$, respectively. Therefore, we derive from \eqref{const} the relations
\beq\label{alphaABc}
    \frac{(A - c)^2}{2} + A = \al = \frac{(B - c)^2}{2} + B,
\eeq
in particular,
\beq\label{ABc}
    (A - B) \Big(1 + \frac{A+B}{2} - c\Big) = 0, \quad\text{hence}\quad A = B \quad\text{or}\quad c = 1 + \frac{A+B}{2}.
\eeq

\begin{remark} In case $A=B$ there are plenty of continuous solutions for the profile function $W$. In fact, the constant $A$ clearly is one, but also the peakon $p(x,t) + A$ with $A = (3c - 4)/3$ and $p$ as in \eqref{peakon}, and many more solitary wave solutions are given in \cite{CLH2012}. 
\end{remark}

As a further observation, to be made use of later, we note that from the boundary condition \eqref{BC} and the boundedness of $W$ we may deduce that
\beq\label{W'BC}
\lim_{\xi \to \pm \infty} W'(\xi) = 0,
\eeq
since the rule of de l'Hospital implies  
$\lim\limits_{\xi \to \pm \infty} W(\xi) = 
    \lim\limits_{\xi \to \pm \infty} \frac{e^{-\xi}W(\xi)}{e^{-\xi}} = 
    \lim\limits_{\xi \to \pm \infty} (W(\xi) - W'(\xi))$.

\subsection{Transformation to a first-order system of differential equations}

According to Proposition \ref{propweaksol} we may make use of \eqref{secondorderpreliminary} to construct a weak traveling wave solution by patching together two pieces of solutions, say, $W_1$ defined on $\xi < 0$ and $W_2$ defined on $\xi > 0$,  such that  the jump conditions \eqref{RH} and \eqref{dercond} at $\xi =0$ are satisfied. Therefore,  we may extract from \eqref{secondorderpreliminary} the second-order differential equation
\beq\label{secondorderdiffequ}
   \frac{(W - c)^2}{2} -  (W')^2 - (W-c) W''   + W = \al
\eeq
and consider pieces of solutions that are defined (at least) on closed half lines  $]-\infty, 0]$ or $ [0, \infty[$.

Note that a shift $\xi \mapsto \xi - \xi_0$ in the independent variable does not alter the structure of the various equations for $W$ considered so far (once the notation $[f]$ is adapted to jumps at $\xi_0$ and \eqref{WODE1} is required for $\xi \neq \xi_0$), hence we may always apply a shift to any appropriate smooth solution piece defined on some half line in order to produce a part of the prospective traveling wave profile to be patched at $\xi = 0$.

Let $W \col I \to \R$ denote a solution to \eqref{secondorderdiffequ}, where $I = \,]-\infty, 0]$ or $I = [0, \infty[$. We will see below how to remove the factor $W(\xi) - c$ in front of the second order derivative in \eqref{secondorderdiffequ}, if we require
\beq\label{Wversusc}
    (\forall \xi \in I \col W(\xi) < c) \quad\text{or}\quad 
    (\forall \xi \in I \col W(\xi) > c).
\eeq
Before doing so, we briefly discuss the situations where these assumptions are not met in the following
\begin{remark}\label{WequalcRemark} Apart from the trivial case with $W$ being the constant solution $c$ (implying thus also $c = A = B$) we have the following instances where $W$ takes on the value $c$ at some point:
\begin{trivlist}

\item{(i)}  If  $W(0-) = c$ or $W(0+) = c$,  then $W$ has to be continuous, for otherwise we obtain a contradiction in the Rankine-Hugoniot condition \eqref{RH} stating $W(0+) + W(0-) = 2 c$.

\item{(ii)} If $W(\xi) = c$ for some $\xi \neq 0$, then evaluation of  \eqref{secondorderpreliminary} at $\xi$ implies $\al = c - W'(\xi)^2$, hence 
$$
   \al \leq c. 
$$
We obtain then from (\ref{alphaABc}-\ref{ABc}) also that  $|A - B| \leq 2$, since $2c = 2 + A + B$ (in case $A \neq B$) yields
$$ 
   0 \leq 2 c - 2 \al = 2 + A + B - ((A - c)^2 + 2 A) = 
    2 + A + B - \big(\frac{(A-B)^2}{4} + 1 + A + B\big) 
    = 1 - \frac{(A-B)^2}{4}.
$$

\end{trivlist}
\end{remark}

In each of the two cases in \eqref{Wversusc} we attempt the change of coordinate $\xi = h(z)$, where  $h \col J \to I$ with $J = \,]-\infty, 0]$ or $J = [0, \infty[$ is the $C^3$ function determined from of the initial value problem
\begin{align}
   h'(z) &= W(h(z)) - c,\label{hODE}\\ 
   h(0) &= 0. \label{h0}
\end{align}

\begin{lemma}\label{translemma} Suppose $c \neq A$, $c \neq B$, and consider $h$ given by (\ref{hODE}-\ref{h0}), then the following hold:
\begin{trivlist} 

\item{(a)} In the first case of \eqref{Wversusc}, $W < c$, we may put $J = - I$ and obtain that $h$ is strictly decreasing and bijective.

\item{(b)} In the second case of \eqref{Wversusc}, $W > c$, we may put $J = I$ and obtain that $h$ is strictly increasing and bijective. 
\end{trivlist}
\end{lemma}
\begin{proof} We discuss the details in case (b) and for the subcsae $I = [0,\infty[$ only, since the proof for the other configurations is analogous except for obvious sign changes. 

From $c \neq B$ and from the boundary condition $\lim_{\xi \to \infty} W(\xi) = B$ we deduce for the $C^2$ function $W$ that $G(\xi) : = \int_0^\xi dy / (W(y) -c)$ is finite for every $\xi > 0$ while $\int_0^\infty dy / (W(y) -c) = \infty$, i.e.,  $\lim_{\xi \to \infty} G(\xi) = \infty$. Thus, the function $G \col [0,\infty[ \to [0,\infty[$ is $C^3$, strictly increasing, and bijective with $G(0) = 0$. From (\ref{hODE}-\ref{h0}) we obtain upon division, integration, and a change of variables that $h(z) = G^{-1}(z)$ for every $z \geq 0$.
\end{proof}

We put $U(z) := W(h(z))$ and $V(z) := W'(h(z))$, then we have
$$ 
U'(z) = W'(h(z)) h'(z) = W'(h(z)) \big(W(h(z)) - c\big) = V(z) \big(U(z) - c\big)
$$
and, with a view on \eqref{secondorderdiffequ}, also
\begin{multline*}
   V'(z) = W''(h(z)) h'(z) = W''(h(z)) \big(W(h(z)) - c\big)\\ =
    \frac{(W(h(z)) - c)^2}{2} -  W'(h(z))^2   + W(h(z)) - \al =
    \frac{(U(z) - c)^2}{2} -  V(z)^2   + U(z) - \al \\
    = -  V(z)^2 + \frac{U(z)^2}{2} + (1-c) U(z) + \frac{c^2}{2} - \al.
\end{multline*}
Thus, we have obtained the following first-order system of ordinary differential equations
\beq\label{firstordersystem}
    \begin{pmatrix} U \\ V\end{pmatrix}' =
   \begin{pmatrix} (U - c) V \\
      - V^2 + \frac{U^2}{2} + (1-c) U + \frac{c^2}{2} - \al
   \end{pmatrix}  =: F(U,V),
\eeq
which is equivalent to \eqref{secondorderdiffequ} under the conditions $c \neq A$ and $c \neq B$ as in Lemma \ref{translemma} for solutions restricted to any of the half planes $U < c$ or $U > c$, because $W$ can be recovered from $U$ via
\beq\label{UviahtoW}
   h(z) := \int_0^z (U(r) - c) \dd r \quad\text{and}\quad W(\xi) := U(h^{-1}(\xi)).
\eeq
We have to keep in mind that, due to the strictly decreasing transformation $\xi = h(z)$ in the case $W < c$ above, the trajectories of \eqref{secondorderdiffequ} in the half plane $U < c$ ``run backward'' in relation to the original part $\xi \mapsto W_1(\xi)$ of the solution in $\xi < 0$.

\begin{remark}\label{rembdandH} \begin{trivlist}
\item (i) Regarding boundary conditions at infinity, that is, if $U$ and $V$ are defined on some interval unbounded above or below, we obtain from \eqref{BC} and the properties of $h$ that
\beq\label{BC1}
   \lim_{z \to \mp\infty} U(z) = \lim_{\xi \to \pm \infty} W(\xi) \text{ in case } U < c,
  \text{ and } \lim_{z \to \mp\infty} U(z) = \lim_{\xi \to \mp \infty} W(\xi) 
  \text{ in case } U > c.
\eeq
And for the $V$-component we deduce 
\beq\label{BC2}
   \lim_{z \to \pm\infty} V(z) = 0
\eeq
directly from \eqref{W'BC}.

\item (ii) The function $H \col \R^2 \to \R$, given by (and adapted from \cite{ZT2008})
$$
  H(U,V) = (U-c)^2 \Big( V^2 - \frac{U^2}{4} + \frac{3c - 4}{6} U 
  + \al - \frac{c^2}{4} - \frac{c}{3} \Big),
$$
is constant along the solutions of \eqref{firstordersystem}, i.e., the orbits are subsets of the level sets of $H$, as can be verified by direct computation. 
\end{trivlist}
\end{remark}

Below we are going to study a little of the qualitative properties of \eqref{firstordersystem} which will enable us to construct a discontinuous wave profile $W$ as indicated in the beginning of the current subsection. To outline the construction in more detail, suppose for example that $B < c < A$ holds---this is a situation to be considered later on, see  \eqref{strongerABc} and the discussion preceding it---, then we have $W(\xi) > c$ for $\xi$ ``near $-\infty$'' and $W(\xi) < c$ for $\xi$  ``near $\infty$'' from the boundary conditions \eqref{BC}. Patching up the solution $W$ then requires: 1. Searching for two solutions to \eqref{firstordersystem} in the form $P = (U_1, V_1) \col ]-\infty,b_1] \to \R^2$ and $Q = (U_2, V_2) \col ]-\infty, b_2] \to \R^2$  which satisfy 
\begin{align}
   \lim_{z \to -\infty} U_1(z) = A, & & U_1(b_1) + U_2(b_2) = 2 c, & &
   \lim_{z \to -\infty} U_2(z) = B, \label{Ucond} \\
   \lim_{z \to -\infty} V_1(z) = 0, & & V_1(b_1) + V_2(b_2) = 0, & &
   \lim_{z \to -\infty} V_2(z) = 0 \label{Vcond}
\end{align}
(recall that the conditions on $U_2$ and $V_2$ account for the ``backward running'' in the region $U < c$). 2. A backtransformation via $z= h^{-1}(\xi)$ as in \eqref{UviahtoW} of appropriately shifted versions of $z \mapsto U_1(z)$ and $z \mapsto U_2(z)$ as patches for $\xi \mapsto W(\xi)$. In this process, the conditions in \eqref{Ucond}  imply that the original boundary conditions \eqref{BC} as well as the Rankine-Hugoniot condition \eqref{RH} are satisfied by $W$, while \eqref{Vcond} guarantee that also \eqref{dercond} and \eqref{W'BC} hold for $W'$.

The relevant equilibrium points of the vector field $F$ in \eqref{firstordersystem} are determined from
$$
    (U - c) V = 0 \quad\text{and}\quad   
    V^2 = \frac{U^2}{2} + (1-c) U + \frac{c^2}{2} - \al
$$
with the additional restriction to $U \neq c$ in case of our intended construction of wave profile functions. (Note that $U=c$ implies $V^2 = c - \al$, which gives again the necessary condition $\al \leq c$ found already in Remark \ref{WequalcRemark}(ii).) We obtain then the two solutions $S_- := (U_0^-,0)$ and $S_+ := (U_0^+,0)$ with 
$$
   U_0^{\pm} = c-1 \pm \sqrt{1 + 2 (\al -c)},
$$
provided that $1 + 2 (\al -c) \geq 0$.

Recall from (\ref{alphaABc}-\ref{ABc}) that in case $A \neq B$ we have $2 c = 2 + A+B$, hence $2 \al =  A + B + 1 + \frac{(A-B)^2}{4}$ and therefore
$$
   U_0^{\pm} = \frac{A+B}{2} \pm \frac{|A-B|}{2},
$$
which means that $U_0^- = \min(A,B)$ while $U_0^+ = \max(A,B)$. 

\beq\label{ABCond}
\text{In all further analysis we focus on the case $A > B$, hence   $U_0^- = B$ and $U_0^+ = A$.}
\eeq

The Jacobian of the vector field $F$ is 
$$
    D F (U,V) = \begin{pmatrix}
       -c V & U - c \\
       U + 1 - c & - 2 V
    \end{pmatrix},
$$
which defines linearizations of the system at the eqilibrium points $S_- = (B,0)$ and $S_+ = (A,0)$ with the respective constant coefficient matrices
$$
      L_-  = \begin{pmatrix}
       0 & -1 - \frac{A-B}{2} \\
       - \frac{A-B}{2} & 0
    \end{pmatrix} \quad\text{and}\quad
     L_+  = \begin{pmatrix}
       0 & -1 + \frac{A-B}{2} \\
       \frac{A-B}{2} & 0
    \end{pmatrix}.
$$
The eigenvalues of $L_-$ are
$$
    \la_1 := - \frac{1}{2}\sqrt{(A-B)(2 + A-B)} < 0 < 
      \frac{1}{2}\sqrt{(A-B)(2 + A-B)} =: \la_2,
$$
hence we have a saddle at $S_- = (B,0)$. We have the eigenvectors 
$$
    r_1 := \begin{pmatrix} \sqrt{2 + A-B}\\ \sqrt{A-B}\end{pmatrix},  
    r_2 := \begin{pmatrix} \sqrt{2 + A-B}\\ - \sqrt{A-B}\end{pmatrix} 
$$    
for $\la_1$, $\la_2$, respectively.

The eigenvalues $\mu$ of $L_+$ are determined from $\mu^2 = (A-B)(A-B - 2)/4$. If we 
strengthen \eqref{ABCond} to the requirement
\beq\label{strongerAB}
    A > B + 2,
\eeq
then $S_+ = (A,0)$ is a saddle point as well, since
$$
    \mu_1 := - \frac{1}{2}\sqrt{(A-B)(A-B - 2)} < 0  <
     \frac{1}{2}\sqrt{(A-B)(A-B - 2)}  =: \mu_2.
$$
In this case, there are the eigenvectors 
$$
   s_1 := \begin{pmatrix} \sqrt{A-B - 2}\\ - \sqrt{A-B}\end{pmatrix},     
     s_2 := \begin{pmatrix} \sqrt{A-B - 2}\\ \sqrt{A-B}\end{pmatrix}
$$ 
for the eigenvalues $\mu_1$, $\mu_2$, respectively.

\begin{remark}  Recall from Remark \ref{WequalcRemark}(ii) that the inequality \eqref{strongerAB} also ensures that we cannot have $W(\xi) = c$, thus  supporting the separation into the open half planes $U < c$ and $U > c$.
\end{remark}

We observe that \eqref{strongerAB} in combination with \eqref{ABc}, i.e., $c = 1 + \frac{A + B}{2}$, gives the refined condition
\beq\label{strongerABc}
   B + 2 <  c  < A.
\eeq
In particular, we see that $S_- = (B,0)$ lies in the left half plane $U < c$, while the saddle point $S_+ = (A,0)$ belongs to the region with $U > c$. A prospective discontinuous wave profile function thus has to be constructed from a trajectory $P$ in the right half plane ``emerging at $z = -\infty$'' from $(A,0)$ with a jump to a trajectory $Q$ in the left half plane  connecting to $(B,0)$ asymptotically, where the points of ``departure'' from $P$ and of ``arrival'' on $Q$ have to be chosen such that the middle parts in the conditions (\ref{Ucond}-\ref{Vcond}) are satisfied.

\medskip

\begin{center}
\includegraphics[width=0.8\textwidth]{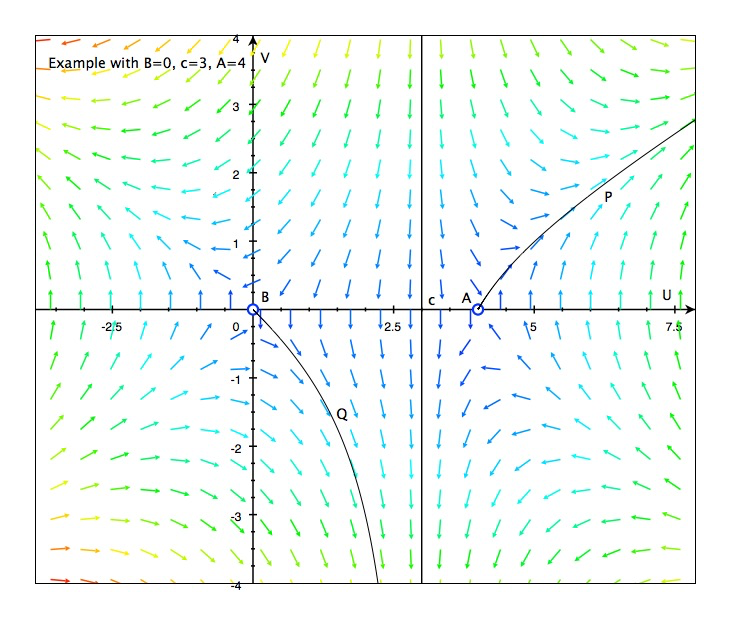}
\end{center}

\subsection{Existence of a discontinuous traveling wave as weak solution}

We suppose that \eqref{strongerABc} holds, i.e., $B + 2 < c < A$, such that we have saddle points at $S_-=(B,0)$ and $S_+=(A,0)$ for the dynamics according to \eqref{firstordersystem} as discussed in the previous subsection. From the eigenvectors $r_2$ and $s_2$ corresponding to the positive eigenvalues in each of the saddle points, we see that there is a unique trajectory $P$ with $\lim_{z \to - \infty} P(z) = (A,0)$ leaving at $(A,0)$ in the direction up (growing $V$) and to the right (growing $U$) and a unique trajectory $Q$ with $\lim_{z \to - \infty} Q(z) = (B,0)$ leaving at $(B,0)$ in the direction down (decreasing $V$) and to the right (growing $U$).

\begin{lemma}\label{intersectionlemma} Suppose that \eqref{strongerABc} holds and 
let $P = (U_1,V_1)$ and $Q = (U_2, V_2)$ be the trajectories defined above. There are unique parameter values $b_1, b_2 \in \R$, such that 
\beq\label{reflection}
        U_1(b_1) + U_2(b_2) = 2c \quad\text{and}\quad V_1(b_1) + V_2(b_2) = 0
\eeq
hold.
\end{lemma}

\begin{proof} We proceed in several steps, proving first separate claims for $P$ and $Q$.

\medskip
\noindent\emph{Claim 1: $U_1$ and $V_1$ are both strictly increasing.}

The definition of $P = (U_1,V_1)$ implies that 
\beq\label{Pregion}
   U_1(z) > A \quad\text{and}\quad V_1(z) > 0
\eeq
holds for $z$ at least in some interval of the form $]-\infty,z_0[$. For every $z \in \R$ satisfying \eqref{Pregion}, the first line in \eqref{firstordersystem} reads  
$$
   U_1'(z) = (U_1(z) - c)\, V_1(z) > (A - c) V_1(z) > 0,
$$
hence $U_1$ is strictly increasing and the first condition in \eqref{Pregion} stays valid. 

The second line of system \eqref{firstordersystem} and \eqref{alphaABc} give
$$
   V_1'(z) = - V_1(z)^2 + \frac{U_1(z)^2}{2} + (1-c) U_1(z) + \frac{c^2}{2} - \al 
   = - V_1(z)^2 + \frac{U_1(z)^2}{2} + (1-c) U_1(z) - \frac{A^2}{2} + (c-1) A.
$$
From Remark \ref{rembdandH} we also know that $P(z)$ lies on the level set of the function $H$ for the value $H(A,0)$, which implies that the relation
\begin{multline}\label{Hrelation}
      \big(U_1(z)-c\big)^2 \Big( V_1(z)^2 - \frac{U_1(z)^2}{4} + \frac{3c - 4}{6} U_1(z) 
  + \al - \frac{c^2}{4} - \frac{c}{3} \Big)\\ = 
  \big(A - c\big)^2 \Big(\frac{A^2}{4} + \Big(\frac{1}{3} - \frac{c}{2}\Big) A + 
  \frac{c^2}{4} - \frac{c}{3}\Big)
\end{multline}
holds. By \eqref{Pregion} we have $U_1(z) - c > A - c$, hence the reverse inequality holds for the second factors in the above equation, i.e., 
$$
   V_1(z)^2 - \frac{U_1(z)^2}{4} + \frac{3c - 4}{6} U_1(z) 
  + \al - \frac{c^2}{4} - \frac{c}{3} < \frac{A^2}{4} + 
  \Big(\frac{1}{3} - \frac{c}{2}\Big) A + \frac{c^2}{4} - \frac{c}{3},
$$ 
which, upon calling again on \eqref{alphaABc}, we may rewrite in the form 
$$
   V_1(z)^2 < \frac{U_1(z)^2}{4} - \frac{3c - 4}{6} U_1(z) - \frac{A^2}{4} 
   + \Big(\frac{c}{2} - \frac{2}{3}\Big) A.
$$
Thus, we have a lower bound for $-V_1(z)^2$ that we insert in the equation for $V_1'(z)$ and find (using the condition \eqref{Pregion} towards the end of the following estimate)
\begin{multline*}
   V_1'(z) > -  \frac{U_1(z)^2}{4} + \frac{3c - 4}{6} U_1(z) + \frac{A^2}{4} 
   - \Big(\frac{c}{2} - \frac{2}{3}\Big) A 
   + \frac{U_1(z)^2}{2} + (1-c) U_1(z) - \frac{A^2}{2} + (c-1) A\\
   = \frac{U_1(z)^2}{4} + \Big(\frac{1}{3} - \frac{c}{2}\Big) U 
     - \frac{A^2}{4} - \Big(\frac{1}{3} - \frac{c}{2}\Big) A =
     \frac{U_1(z)^2 - A^2}{4} + \Big(\frac{1}{3} - \frac{c}{2}\Big) (U - A) \\
     = (U - A) \Big(\frac{U_1(z) + A}{4}  + \frac{1}{3} - \frac{c}{2}\Big) >
     (U - A) \Big( \frac{A + A}{4}  + \frac{1}{3} - \frac{c}{2} \Big) =
     (U - A) \Big( \frac{A -c}{2}  + \frac{1}{3} \Big) > 0.
\end{multline*}
Hence we see that $V_1$ is strictly increasing and the conditions in \eqref{Qregion} remain valid  throughout.

\medskip
\noindent\emph{Claim 2: Let $]-\infty,p[$ be the maximal interval of existence for the solution $P$ (regardless whether $p$ is finite or $p = \infty$, though we conjecture the latter), then 
$U_1(z) \to \infty$ and $V_1(z) \to \infty$ as $z \to p$.}

We recall that $P$ stays entirely in the region \eqref{Pregion}, where in particular $U_1 > A > c$. Hence the norm $\Norm{P(z)}$ has to become unbounded as $z \to p$, since the solution cannot reach the boundary points of the domain along $U = c$. Thus, at least one of the components $U_1(z)$ or $V_1(z)$ is unbounded as $z \to p$. But observing $U_1(z) - c > A - c > 0$ we may deduce from \eqref{Hrelation} that either both component functions, $U_1(z)$ and $V_1(z)$,  stay bounded or both are unbounded as $z \to p$. We conclude that both have to be unbounded.

\medskip
\noindent\emph{Claim 3: The function $U_2$ is strictly increasing, while $V_2$ is strictly decreasing. Moreover, the maximal interval of existence for the solution curve $Q$ is of the form $]-\infty, q[$ with some $q \in \R$ and, as $z \to q$, we have $V_2(z) \to -\infty$ and $U_2(z) \to c$.}

By definition of $Q = (U_2,V_2)$, we have 
\beq\label{Qregion}
   B < U_2(z) < c \quad\text{and}\quad V_2(z) < 0
\eeq
for $z$ at least in some interval bounded only on the right (note that by connectedness, $U < c$ has to hold for every solution starting somewhere in the left half plane). For every $z \in \R$ satisfying \eqref{Qregion}, we immediately deduce from the first line in \eqref{firstordersystem} that 
$$
   U_2'(z) = (U_2(z) - c)\, V_2(z) > 0,
$$
hence $U_2$ is strictly increasing and the first condition in \eqref{Qregion} will continue to be valid. 

From the second line in \eqref{firstordersystem} we obtain
$$
   V_2'(z) = - V_2(z)^2 + \frac{U_2(z)^2}{2} + (1-c) U_2(z) + \frac{c^2}{2} - \al 
   = - V_2(z)^2 + \frac{\big(U_2(z) + 1 - c\big)^2 - 1}{2}  + c - \al.
$$
We note that \eqref{Qregion} and \eqref{ABc} imply $1 = c + 1 - c >  U_2(z) + 1 - c > B + 1 - c = \frac{B - A}{2}$ and by \eqref{strongerABc} we then have
$$
     | U_2(z) + 1 - c| \leq \max (1, \frac{A - B}{2}) = \frac{A - B}{2}.
$$
Moreover,  \eqref{alphaABc} gives $c - \al = \frac{1}{2} - \frac{(A - B)^2}{8}$
 and therefore,
$$
    V_2'(z) \leq - V_2(z)^2 + \frac{\frac{(A - B)^2}{4} - 1}{2}  
    + \frac{1}{2} - \frac{(A - B)^2}{8} = - V_2(z)^2 < 0,
$$
which implies that $V_2$ is strictly decreasing and the second condition in \eqref{Qregion} will hold throughout. In particular, we obtained the inequality $-V_2' / V_2^2 \geq - 1$,  which we may integrate over an interval $[z_0, z]$ to obtain the chain of inequalities
$$
   0 > \frac{1}{V_2(z)}  \geq \frac{1}{V_2(z_0)} + z - z_0,
$$ 
which implies $z < z_0 - 1/V_2(z_0)$ and that $V_2(z) \to - \infty$ when $z$ approaches the upper bound. Therefore, the maximal interval of existence of $Q$ is of the form $]-\infty,q[$ with finite $q \in \R$.

By boundedness and monotonicity of $U_2$, there exists $c_0 := \lim_{z \to q}U_2(z)$ satisfying $B < c_0 \leq c$. We may argue as in \eqref{W'BC} to see that $\lim_{z \to q}U_2'(z) = 0$. On the other hand, $U_2'(z) = (U_2(z) - c) V_2(z)$ would necessarily tend to $+\infty$ (as $z \to q$), unless $c_0 = c$.

\medskip

Combining now the information from the claims proved above, we complete the proof by the following observation: If $\tilde{P}$ denotes the pointwise reflection of $P$ at $(c,0)$, i.e., $\tilde{P}(z) = 2 (c,0) - P(s)$, then $Q$ and $\tilde{P}$ have a unique intersection point (in the region $B < U < c$, $V < 0$), which corresponds to a unique parameter value $b_2$ along $Q$ and to a unique parameter value $b_1$ along $P$. The condition of reflection at $(c,0)$ reproduces precisely the relation \eqref{reflection}.
\end{proof}

The above lemma shows that we can indeed find solutions to \eqref{firstordersystem} satisfying (\ref{Ucond}-\ref{Vcond}), which can therefore be used as solution patches for a discontinuous wave profile $W$.

\begin{theorem} If $B+2 < c < A$, then there exists a discontinuous wave profile $W$ defining a weak traveling wave solution $u$ to the Cauchy problem (\ref{FWEqu}-\ref{IC}) with initial value $u_0 = W$. 
\end{theorem}

\begin{proof} 
Let $U_1$ and $U_2$ be the first components of the solution curves $P$ and $Q$ introduced above and let $b_1$, $b_2$ be as in Lemma \ref{intersectionlemma}. Note that the hypothesis $B+2 < c < A$ implies that Lemma \ref{translemma} and the subsequent transformations between the second-order equation for $W$ and the first-order system for $(U,V)$ are applicable and preserve equivalence. 
Construct $h$ from $U$ as in \eqref{UviahtoW} and put $W(\xi) := U_1 (h^{-1}(\xi) + b_1)$, if $\xi < 0$, and $W(\xi) := U_2(h^{-1}(\xi) + b_2)$, if $\xi > 0$. From (\ref{Ucond}-\ref{Vcond}) and the observations in the discussion of these conditions, we see that the proof is complete by appealing to the second statement in Proposition \ref{propweaksol}.
\end{proof}

These specific discontinuous traveling wave solutions serve here more as a mathematical test case for the weak solution concept and will not be useful as models of water wave profiles. One might see them as a  reminiscence of shock wave solutions for the Burgers equation in its nonlocal perturbation described by 
\eqref{FWEqu2}, or rather in its weak form by \eqref{weaksol}.


\bibliography{G2018}
\bibliographystyle{abbrv}

\end{document}